\newcommand{\coloneqq}{\mathrel{\mathop:}=}

\documentclass[brochure,english,12pt]{smfbourbaki}

\usepackage[T1]{fontenc}
\usepackage{lmodern,amssymb,bm}

\usepackage{babel}

\usepackage[utf8]{inputenc}

\usepackage[colorlinks=true, linkcolor=blue, citecolor=red,
urlcolor=blue]{hyperref}
\usepackage[
backend=biber,
style=authoryear, 
citestyle=authoryear-comp,
maxnames=7,
sortcites=false %%%% to keep the order in \cite command
]{biblatex}
\usepackage{csquotes}

\usepackage{tikz}
\usepackage[all]{xy}
\SelectTips{cm}{}
\def\R{\mathbb{R}}
\def\N{\mathbb{N}}
\def\Q{\mathbb{Q}}
\def\Z{\mathbb{Z}}

\def\fa#1{\forall_{#1}\;\;\;}
\def\exi#1{\exists_{#1}\;\;\;}

\DeclareMathOperator{\FG}{FG}
\DeclareMathOperator{\EGR}{Exp}
\DeclareMathOperator{\eGR}{e}
\def\egrs#1#2{\eGR(#1,#2)}
\def\egr#1{\eGR(#1)}

\DeclareMathOperator{\ord}{ord}
\def\egrord#1{\ord_{\EGR}(#1)}%{\zeta_{\GR}(#1)} % better: \omega instead of \zeta

\DeclareMathOperator{\Aut}{Aut}
\def\gr#1#2#3{\beta_{#3}(#1,#2)}
\def\Gr#1#2#3{B_{#3}(#1,#2)}
\def\genrel#1#2{\langle #1 \mid #2 \rangle}
\DeclareMathOperator{\rk}{rk}

\def\sv#1{\| #1 \|}
\DeclareMathOperator{\SV}{SV}
\def\SVG#1#2{\SV_{#1}(#2)}
\DeclareMathOperator{\vol}{vol}

\def\actson{\curvearrowright}

\DefineBibliographyExtras{french}{\restorecommand\mkbibnamefamily}

\DeclareDelimFormat{nameyeardelim}{\addcomma\space}

\DeclareNameAlias{sortname}{given-family}

\renewcommand{\bibnamedash}{\leavevmode\raise3pt\hbox to3em{\hrulefill}\space}

\AtEveryBibitem{%
  \clearfield{issn} % Remove issn
  \clearfield{isbn} % Remove isbn
  \clearfield{doi} % Remove doi
  \clearlist{language} %%% remove language
  \ifentrytype{online}{}{% Remove url except for @online
    \clearfield{url}
  }
}

\renewbibmacro{in:}{%
    \ifentrytype{article}{}{\printtext{\bibstring{in}\intitlepunct}}}

\DeclareFieldFormat[article,periodical,inreference]{number}{\mkbibparens{#1}}
\DeclareFieldFormat[article,periodical,inreference]{volume}{\mkbibbold{#1}}
\renewbibmacro*{volume+number+eid}{%
    \printfield{volume}%
    \setunit*{\addthinspace}% NEW
    \printfield{number}%
    \setunit{\addcomma\space}%
    \printfield{eid}}

\DeclareFieldFormat[article,inbook,incollection]{title}{\enquote{#1}\addcomma} 

\date{Avril 2023}
\bbkannee{75\textsuperscript{e} année, 2022--2023}  % Commandes spéciales
\bbknumero{1206}                                      % Commandes spéciales

\title{Exponential growth rates in hyperbolic groups}
\subtitle{after Koji Fujiwara and Zlil Sela}
\author{Clara Löh}
\address{Fakult\"at f\"ur Mathematik\\
  Universit\"at Regensburg\\
  93040 Regensburg\\
  Germany}
\email{clara.loeh@ur.de}

\begin{document}

\maketitle

% abstract
A classical result of J\o rgensen and Thurston shows that the set of
volumes of finite volume complete hyperbolic $3$-manifolds is a
well-ordered subset of the real numbers of order
type~$\omega^\omega$; moreover, each volume can
only be attained by finitely many isometry types of hyperbolic
$3$-manifolds.

\textcite{fs} established a group-theoretic companion of this result:
If $\Gamma$ is a non-elementary hyperbolic group, then the set of
exponential growth rates of~$\Gamma$ is well-ordered, the order type
is at least~$\omega^\omega$, and each growth rate can only be attained
by finitely many finite generating sets (up to automorphisms).

In this talk, we outline this work of Fujiwara and Sela and discuss
related results.

\tableofcontents

%%%%%%%%%%%%%%%%%%%%%%%%%%%%%%%%%%%%%%%%%%%%%%%%%%%%%%%%%%%%%%%%%%%%%%
\section{Main results}\label{sec:results}

Geometric group theory provides a rich interaction between
the Riemannian geometry of manifolds and the large-scale geometry
of finitely generated groups. This bond is particularly strong
in the presence of negative curvature and explains a
variety of rigidity phenomena.  
The group-theoretic analogues of closed hyperbolic manifolds are
hyperbolic groups; more generally, the group-theoretic analogues of
finite volume complete hyperbolic manifolds are relatively hyperbolic
groups. 

The volume growth behaviour of Riemannian balls in the universal
covering of a compact Riemannian manifold is the same as the growth
behaviour of balls in Cayley graphs of the fundamental group. By
definition, the exponential growth rates of finitely generated groups
measure the exponential expansion rate of balls in Cayley graphs and
thus are entropy-like invariants. While there is no direct connection
between the volume of a hyperbolic manifold~$M$ and the exponential
growth rates of~$\pi_1(M)$, the results of \textcite{fs} show that
certain sets of such values share fundamental structural similarities.

To state these results, for a finitely generated group~$\Gamma$, we
write~$\EGR(\Gamma) \subset \R$ for the (countable) set of all
exponential growth rates~$\egrs \Gamma S$ with respect to finite
generating sets~$S$ of~$\Gamma$.  The automorphism
group~$\Aut(\Gamma)$ acts on the set~$\FG(\Gamma)$ of all finite
generating sets of~$\Gamma$ and $\egrs \Gamma {f(S)} = \egrs \Gamma S$
holds for all~$S \in \FG(\Gamma)$ and all~$f \in \Aut(\Gamma)$.  More
details on terminology and notation can be found in
Appendix~\ref{appx:terminology}.
\enlargethispage{\baselineskip} % to ensure that the Overview also fits on this page

\begin{theo}[well-orderedness; \protect{\cite[Theorem~2.2]{fs}}]\label{thm:wellordered}
  If $\Gamma$ is a hyperbolic group, then $\EGR(\Gamma)$
  is well-ordered (with respect to the standard order on~$\R$).
\end{theo}

\begin{theo}[finite ambiguity; \protect{\cite[Theorem~3.1]{fs}}]\label{thm:fin}
  The set~$
  \{ S \in \FG (\Gamma) \mid \egrs \Gamma S = r \}
  / \Aut(\Gamma)
  $
  is finite for every non-elementary hyperbolic group~$\Gamma$
  and every~$r \in \R$.
\end{theo}

\begin{theo}[growth ordinals; \protect{\cite[Proposition~4.3]{fs}}]\label{thm:ord}
  Let $\Gamma$ be a non-elementary hyperbolic group.  Then the ordinal
  number~$\egrord \Gamma$ associated with~$\EGR(\Gamma)$
  satisfies~$\egrord \Gamma \geq \omega^\omega$.
\end{theo}

Moreover, \textcite[Proposition~4.3]{fs} show that $\egrord \Gamma
=\omega^\omega$ if epi-limit groups over~$\Gamma$ have a Krull
dimension.  In analogy with the case of hyperbolic $3$-manifolds,
they conjecture that $\egrord \Gamma =
\omega^\omega$ holds for all non-elementary hyperbolic
groups~$\Gamma$ \parencite[Section~4]{fs} .

\begin{exem}
  If $F$ is a finitely generated free group of rank at least~$2$, then
  limit groups over~$F$ have  a Krull dimension \parencite{louder}.
  Hence, Theorems~\ref{thm:wellordered}--\ref{thm:ord} show that 
  $\egrord F = \omega^\omega$ and each value in~$\EGR(F)$ is realised
  by only finitely many generating sets (up to automorphisms of~$F$).
\end{exem}

The key idea for the proofs of
Theorems~\ref{thm:wellordered}--\ref{thm:ord} is inspired by the
proofs by Thurston and J\o rgensen for the set of volumes of
hyperbolic $3$-manifolds and model theory: One passes from sequences
of generating sets (of bounded size) of the given hyperbolic
group~$\Gamma$ to a limit group over~$\Gamma$ with an associated
finite generating set; i.e., limit groups play the role of cusped
manifolds.  The main challenge is then to compute the exponential
growth rate of this limiting object in terms of the exponential growth
rates appearing in the original sequence.

\subsection*{Overview}

Basics on hyperbolic groups, exponential growth rates, and
well-ordered countable sets are recalled in
Appendix~\ref{appx:terminology}.  We briefly explain the manifold
context of the results above in Section~\ref{sec:context}, with a
focus on hyperbolic and simplicial volume.  Section~\ref{sec:proof}
gives a proof outline of the main results.  Finally, in
Section~\ref{sec:apps}, we mention applications and extensions of the
main results.

%%%%%%%%%%%%%%%%%%%%%%%%%%%%%%%%%%%%%%%%%%%%%%%%%%%%%%%%%%%%%%%%%%%%%%
\section{Context: Volumes of manifolds and hyperbolicity}\label{sec:context}

The results of \textcite{fs} are analogues of the behaviour of volumes
of finite volume complete hyperbolic $3$-manifolds.  We recall this
background in Section~\ref{subsec:hyp}. The situation for simplicial
volume is discussed in Section~\ref{subsec:sv}. In
addition, we mention right-computability as a further
structural property of ``volume'' sets (Section~\ref{subsec:rc}).

%%%%%%%%%%%%
\subsection{Hyperbolic volume}\label{subsec:hyp}

The structure and volumes of hyperbolic $3$-manifolds was
analysed in the breakthrough work of J\o rgensen and Thurston. 

\begin{theo}[volumes of hyperbolic $3$-manifolds; \protect{\cite[Chapter~6]{thurston}}]
  The set
  \[ \{ \vol(M)
  \mid \text{$M$ is a finite volume complete hyperbolic $3$-manifold}
  \} 
  \]
  is well-ordered (with respect to the standard order on~$\R$)
  and the associated ordinal is~$\omega^\omega$.
  Moreover, 
  every value arises only from finitely many isometry
  classes of finite volume hyperbolic $3$-manifolds.
\end{theo}

We briefly summarise the main steps of the proof
\parencite{gromov_bourbaki}; the key is to study the convergence of
sequences of hyperbolic manifolds and to understand the role of
hyperbolic manifolds with cusps as limits of such sequences:
\begin{enumerate}
\item Every sequence~$(M_n)_{n \in \N}$ of complete hyperbolic $3$-manifolds
  with uniformly bounded volume contains a subsequence that
  converges in a strong geometric sense to a finite volume complete
  hyperbolic $3$-manifold~$M$ and $\lim_{n \to\infty} \vol(M_n) = \vol(M)$.
  Furthermore, for ``non-trivial'' such sequences, one can show that
  $\vol(M) > \vol(M_n)$ holds for all members~$M_n$ of the subsequence.

  This can be used to show that the set of hyperbolic volumes is well-ordered  
  and that every value can only be obtained in finitely many ways.
\item Every finite volume complete hyperbolic $3$-manifold with $k \in
  \N$ cusps can be obtained for each~$p \in \{0,\dots,k\}$ as the
  limit of a sequence of finite volume complete hyperbolic
  $3$-manifolds with exactly $p$ cusps.

  This can be used to show that the volume ordinal is at least~$\omega^k$.
  Constructing hyperbolic $3$-manifolds with arbitrarily large numbers
  of cusps thus shows that the volume ordinal is at least~$\omega^\omega$.
  In combination with the first part, one can derive that the volume ordinal
  equals~$\omega^\omega$.
\end{enumerate}

In contrast, in higher dimensions, the set of volumes of
finite volume complete hyperbolic manifolds leads to the
ordinal~$\omega$. This follows from Wang's finiteness theorem
and the unboundedness of hyperbolic volumes.

\begin{theo}[Wang's finiteness theorem; \cite{wang}]
  Let $n \in \N_{\geq 4}$ and $v \in \R_{\geq 0}$. Then
  there exist only finitely many isometry classes of
  finite volume complete hyperbolic $n$-manifolds~$M$
  with~$\vol (M) \leq v$.
\end{theo}

%%%%%%%%%%%%
\subsection{Simplicial volume}\label{subsec:sv}

Simplicial volume is a homotopy invariant of closed
manifolds. For several geometrically relevant classes
of Riemannian manifolds, the simplicial volume encodes
topological rigidity properties of the Riemannian volume.

\begin{defi}[simplicial volume; \cite{gromov_vbc}]
  The \emph{simplicial volume} of an oriented closed connected
  manifold~$M$ is the $\ell^1$-semi-norm of its (singular)
  $\R$-fundamental
  class:
  \[ \sv M
  \coloneqq  \| [M]_\R \|_1
  \coloneqq  
  \inf
  \bigl\{ \sum_{j=1}^k |a_j|
  \bigm| \sum_{j=1}^k a_j \cdot \sigma_j
         \text{ is a singular $\R$-fundamental cycle of~$M$}
  \bigr\}
  \]
\end{defi}

For genuine hyperbolic manifolds, the simplicial volume
leads to the same ordering and finiteness behaviour
as the hyperbolic volume (Section~\ref{subsec:hyp}):

\begin{exem}[hyperbolic manifolds]
  If $M$ is an oriented closed connected hyperbolic manifold
  of dimension~$n$, then
  \[ \sv M = \frac{\vol(M)}{v_n},
  \]
  where $v_n \in \R_{>0}$ is the hyperbolic volume of ideal
  regular geodesic $n$-simplices in hyperbolic $n$-space
  \parencite{thurston,bp}. A similar relationship
  also holds in the complete finite volume case
  (\cite{thurston}; \cite[Appendix~A]{fujiwaramanning}).
  In particular, this proportionality can be used to prove
  mapping degree estimates in terms of the hyperbolic volume
  for continuous maps between hyperbolic manifolds
  \parencite{gromov_vbc}.
\end{exem}

Passing to the setting of fixed hyperbolic fundamental groups, we
obtain:

\begin{exem}[hyperbolic fundamental group]
  Let $\Gamma$ be a finitely presented group and let $n \in \N$. Then
  the set
  \[ \SVG \Gamma n
  \coloneqq  \{ \sv M
  \mid \text{$M$ is an oriented closed connected $n$-manifold
             with~$\pi_1(M) \cong \Gamma$}
  \}
  \]
  is a subset of~$\{ \| \alpha\|_1 \mid \alpha \in H_n(\Gamma;\R)
  \text{ is integral}\}$, where a class in~$H_n(\Gamma;\R)$ is
  \emph{integral} if it lies in the image of the change of
  coefficients map~$H_n(\Gamma;\Z) \to H_n(\Gamma;\R)$
  \parencite[Section~3.1]{loeh_simvolpi1}.

  If $\Gamma$ is hyperbolic and $n \geq 2$, then $\|\cdot\|_1$
  is a norm on~$H_n(\Gamma;\R)$ (by the results of \textcite{mineyev}
  on bounded cohomology and the duality principle).
  In particular: The set~$\SVG \Gamma n \subset \R$ is well-ordered
  and for~$n\geq 4$ the ordinal associated with~$\SVG \Gamma n$ is
  \begin{itemize}
  \item
    either~$0$ (if~$H_n(\Gamma;\R) \cong 0$);
  \item
    or~$\omega$ (if $H_n(\Gamma;\R) \not\cong 0$): 
    In this case, normed Thom realisation shows that
    indeed infinitely many different values are realised
    \parencite[Section~3.1]{loeh_simvolpi1}.
  \end{itemize}

  For~$n \geq 4$, finite ambiguity breaks down in this
  general topological setting:
  If $M$ is an oriented closed connected $n$-manifold, then for
  each~$k \in \N$, the manifold~$M$ and the iterated connected
  sums~$M_k \coloneqq  M \mathbin{\#} (S^2 \times S^{n-2})^{\# k}$ have the
  same simplicial volume \parencite{gromov_vbc} and isomorphic
  fundamental groups. However, the manifolds~$M_0,M_1,\dots$ all
  have different homotopy types (as can be seen from the homology in
  degree~$2$).
\end{exem}

%%%%%%%%%%%%
\subsection{Right-computability}\label{subsec:rc}

In the previous discussion, we focussed on the order
structure of volumes and exponential growth rates. 
Many real-valued invariants in geometric group theory
and geometric topology also carry another, complementary,
structure: They tend to have an intrinsic
limit on their computational complexity. In particular,
such a limit gives additional constraints on the possible
sets of values.

\begin{defi}[right-computable]\label{def:rc}
  A real number~$\alpha$ is \emph{right-computable}
  if the set~$\{x \in \Q \mid x > \alpha \}$ is recursively
  enumerable.
\end{defi}

For example, simplicial volumes of oriented closed connected
manifolds are right-computable real numbers
\parencite{heuerloeh_trans}. On the group-theoretic side,
right-computability naturally arises for stable commutator length
of recursively presented groups \parencite{heuer_sclrc} or $L^2$-Betti
numbers of groups with controlled word problem
\parencite{loehuschold}. Concerning exponential growth rates, we
have the follwing:

\begin{prop}[right-computability of exponential growth rates]
  \label{prop:egr_algorc}
  There exists a Turing machine that
  \begin{itemize}
  \item given a finite presentation~$\genrel SR$ and a
    finite set~$S'$ of words over~$S \sqcup S^{-1}$,
  \item does 
    \begin{itemize}
    \item \emph{not} terminate if $S'$ does not represent
      a generating set of the group~$\Gamma$ described by~$\genrel SR$;
    \item terminate and return an enumeration
      of~$\{ x \in \Q \mid x > \egrs \Gamma {S'} \}$
      if $S'$ represents a generating set of~$\Gamma$.
    \end{itemize}
  \end{itemize}
\end{prop}

\begin{coro}\label{cor:egr_basicrc}
  Let $\Gamma$ be a finitely presented group.
  \begin{enumerate}
  \item For every~$S \in \FG(\Gamma)$, the real number~$\egrs \Gamma S$
    is right-computable. 
  \item For every~$r \in \Q$, the truncated
    set~$\{ S \in \FG(\Gamma) \mid \egrs \Gamma S < r \}$
    is recursively enumerable. 
  \end{enumerate}
\end{coro}

Proofs of these observations are provided in Appendix~\ref{appx:rc}. 
In particular, such results could be used to give a crude a priori
upper bound for~$\egrord \Gamma$ by a ``large'' countable ordinal 
for all finitely presented groups~$\Gamma$ with well-ordered
set~$\EGR (\Gamma)$.

%%%%%%%%%%%%%%%%%%%%%%%%%%%%%%%%%%%%%%%%%%%%%%%%%%%%%%%%%%%%%%%%%%%%%%
\section{Proof technique}\label{sec:proof}

We outline the proofs of Theorem~\ref{thm:wellordered}--\ref{thm:ord}
by \textcite{fs}.  These proofs roughly follow the blueprint of the
case of hyperbolic $3$-manifolds (Section~\ref{subsec:hyp}), where
limit groups will play the role of cusped manifolds:
\begin{itemize}
\item A compactness phenomenon turns convergence of
  exponential growth rates into convergence of actions (of
  subsequences) to the action of a limit group.
\item 
  The main challenge is then to compute the exponential growth rates
  of these limit groups as the limit of the given exponential growth rates.
\end{itemize}
Before going into these arguments, we first recall basic notions
on limit groups.

%%%%%%%%%%%%
\subsection{Limit groups}

Limit groups are groups that arise as ``limits'' -- in various senses
-- of groups. These groups are convenient tools in the model theory of
groups and in geometric group theory
\parencite{km_1,km_2,sela_diophantine6,groveswilton}.  In analogy with
$3$-manifolds, limit groups over hyperbolic groups admit a
JSJ-decomposition (\cite{sela_diophantine7};
\cite[Section~4]{weidmannreinfeldt}).  Limit groups over a given
group~$\Gamma$ are the finitely generated subgroups of non-principal
ultraproducts of~$\Gamma$.  More explicitly:

\begin{defi}[limit group]
  Let $\Gamma$ be a group.
  \begin{itemize}
  \item A \emph{stable homomorphism} from a group~$\Lambda$ to~$\Gamma$
    is a sequence~$(f_n \colon \Lambda \to \Gamma)_{n \in \N}$ of homomorphisms
    with the property
    \[ \fa{x \in \Lambda}
    \exi{N \in \N}
    (\fa{n \in \N_{\geq N}} f_n(x) = 1)
    \lor
    (\fa{n \in \N_{\geq N}} f_n(x) \neq 1).
    \]
    The \emph{stable kernel} of a stable homomorphism~$f_* \colon \Lambda \to \Gamma$
    is defined as
    \[ \ker f_* \coloneqq 
    \bigl\{ x \in \Lambda
    \bigm| \exi{N \in \N} \fa{n \in \N_{\geq N}} f_n(x) = 1
    \bigr\}
    \subset \Lambda.
    \]
  \item A \emph{limit group over~$\Gamma$} is a group of the
    form~$\Lambda/\ker f_*$, where $\Lambda$ is a finitely generated
    group and $f_* \colon \Lambda \to \Gamma$ is a stable
    homomorphism. The canonical projection~$f \colon \Lambda \to
    \Lambda/\ker f_*$ is the \emph{limit homomorphism} and we say that
    $f_*$ \emph{converges to~$f$}.  A limit group over~$\Gamma$ is an
    \emph{epi-limit group over~$\Gamma$} if the~$f_n$ can be chosen to
    be epimorphisms.
  \item A \emph{limit group} is a limit group over a finitely
    generated free group.
  \end{itemize}
\end{defi}

\begin{exem}
  If $\Gamma$ is a group, then every finitely generated subgroup of~$\Gamma$ can be
  viewed as a limit group over~$\Gamma$
  (via the inclusion homomorphisms). In particular, $\Gamma$ is
  an epi-limit group over~$\Gamma$ if $\Gamma$ is finitely generated.
\end{exem}

%%%%%%%%%%%%
\subsection{Limits and their exponential growth rate}\label{subsec:limegr}

\begin{theo}[compactness; \protect{\cite[proof of Theorem~2.2]{fs}}]\label{thm:compactness}
  Let $\Gamma$ be a hyperbolic group, let $(X,d)$ be a Cayley graph
  of~$\Gamma$, and let $(S_n)_{n \in \N}$ be a sequence of finite
  generating sets of~$\Gamma$ such that the sequence~$(\egrs \Gamma
  {S_n})_{n \in \N}$ is bounded. Then there exists a subsequence
  (again denoted by~$(S_n)_{n \in \N}$) with the following properties:
  \begin{itemize}
  \item All~$S_n$ have the same size. Let~$S$ be a set of this cardinality
    and let $F$ be the free group freely generated by~$S$. 
  \item There exist epimorphisms~$f_n \colon F \to \Gamma$ for
    all~$n \in \N$ such that $f_n(S)$ is conjugate to~$S_n$.
    Moreover, $f_*$ is a stable homomorphism~$F \to \Gamma$.
    Let $L$ denote the associated limit group.
  \item The sequence
    \[ \bigl(F
    \actson_{f_n}
    \bigl(X, \frac1{\max_{s \in S} d(1,f_n(s))} \cdot d\bigr)\bigr)_{n \in \N}
    \]
    of actions (induced by~$f_n \colon F \to \Gamma$ and
    the translation action of~$\Gamma$ on~$X$)
    converges in the $F$-Gromov--Hausdorff distance to
    a faithful action of~$L$ on a real tree.
  \end{itemize}
\end{theo}
\begin{proof}
  The boundedness of the exponential growth rates allows us
  to fix the size of the generating set because the exponential
  growth rate grows at least linearly in the size of
  the generating set by an estimate of 
  \textcite{arzhantsevalysenok}. 

  One can then apply the Bestvina--Paulin method after conjugating and
  rescaling appropriately (\cite[proof of Theorem~2.2]{fs};
  \cite[Section~2]{weidmannreinfeldt}).
\end{proof}

\begin{theo}[limits of exponential growth rates; \protect{\cite[Proposition~2.3]{fs}}]
  \label{thm:limegr}
  Let $\Gamma$ be a hyperbolic group, let $\Lambda$ be a finitely
  generated group with finite generating set~$S$, and let $(f_n \colon
  \Lambda \to \Gamma)_{n \in \N}$ be a stable homomorphism consisting
  of epimorphisms that converges to a limit group~$f \colon \Lambda
  \to L$ over~$\Gamma$ with a faithful action on a real tree.  Then
  $\egrs \Gamma {f_n(S)} \leq \egrs L {f(S)}$ holds for all large enough~$n \in \N$
  and 
  \[ \lim_{n \to \infty} \egrs \Gamma {f_n(S)} = \egrs L {f(S)}
  \]
\end{theo}

\begin{proof}[Sketch of proof]
  By precomposition, 
  without loss of generality we may assume that
  $\Lambda$ is free and that $S$ a free generating set. 

  We first explain why ``$\leq$'' holds (provided the limit exists):
  Because $\Gamma$ is hyperbolic, for all large enough~$n \in \N$,
  there exists a homomorphism~$h_n \colon L \to \Gamma$
  with~$f_n = h_n \circ f$
  \parencite[Lemma~6.5, Corollary~7.13]{weidmannreinfeldt}:
  \[
  \xymatrix{%
  (\Lambda,S)
    \ar[d]_-{f}
    \ar[dr]^-{f_n}
  &
    \\
    (L,f(S))
    \ar@{-->}[r]_-{h_n}
    & (\Gamma,f_n(S))
  }
  \]
  Therefore, monotonicity of the exponential growth rates
  (Remark~\ref{rem:mono}) implies that $\egrs \Gamma {f_n(S)} \leq
  \egrs L {f(S)}$ for all large enough~$n \in \N$.

  The hard work lies in proving convergence and~``$\geq$'': 
  Given~$\varepsilon \in \R_{>0}$, the goal is to show that
  for all large enough~$n \in \N$, we have 
  \[ \log \egrs \Gamma {f_n(S)} 
  \geq \log \egrs L {f(S)} - \varepsilon.
  \]

  Matters would be simple if, given~$N \in \N$, the
  multiplication-projection map
  \begin{align*}
    \Gr L {f(S)} N ^q
    & \to \Gr {\Gamma} {f_n(S)} {q \cdot N}
    \\
    (w_1,\dots,w_q)
    & \mapsto h_n(w_1 \cdot \dots \cdot w_q)
  \end{align*}
  were injective for all large enough~$n \in \N$ and all~$q \in
  \N$. However, this will not happen in general. Using the faithful
  limit action of~$L$ on a real tree, \textcite[proof of
    Proposition~2.3]{fs} find enough freeness inside~$L$ to show
  through delicate estimates that
  there exists a~$b \in \N$, a four-element subset~$U \subset \Gr L
  {f(S)} b$, and a constant~$C \in \R_{>0}$ with:
  \begin{itemize}
  \item[]
    For
    all~$q \in \N$, there is a map~$\varphi_{q} \colon L^q
    \to L$ 
    of the form
    \[ (w_1,\dots, w_q) \mapsto
    w_1 \cdot u_1 \cdot \dots \cdot w_q \cdot u_q,
    \]
    where the ``separators''~$u_1,\dots, u_q \in U$ may depend
    on~$w_1,\dots, w_q$ and satisfy a ``small cancellation condition''
    that ensures the following: Given~$N \in \N$,  
    for all large enough~$n \in \N$ and all~$q \in \N$, the
    map~$h_n \circ \varphi_q \colon L^q \to \Gamma$ is injective on
    at least a subset~$A_{N,n,q}$ of size $(1/ C \cdot \gr L {f(S)} N)^q$
    of~$\Gr L {f(S)} N ^q$. In particular,
    \[
    \gr \Gamma {f_n(S)} {q \cdot (N+b)}
    \geq \# A_{N,n,q} 
    \geq \bigl( \frac 1 C \cdot \gr L {f(S)} N \bigr)^q.
    \]
    More specifically, this works for all~$n \in \N$ that are large
    enough so that $h_n$ is injective on~$\Gr L {f(S)} {2 \cdot N}$;
    such~$n$ exist in view of the convergence of actions.
  \end{itemize}
  Given~$\varepsilon \in \R_{>0}$, we choose~$N \in \N$
  large enough to have
  \[ \frac1 {N+b} \cdot
  (\log \gr L {f(S)} N - \log C)
  \geq \frac 1 N \cdot \log \gr L {f(S)} N - \varepsilon.
  \]
  Then, we obtain for all large enough~$n \in \N$ that
  \begin{align*}
    \log \egrs \Gamma {f_n(S)}
    & = \lim_{q \to \infty} \frac 1 {q \cdot (N+b)}
    \cdot \log \gr \Gamma {f_n(S)} {q \cdot (N+b)}
    \\
    & \geq \frac 1 {N+b} \cdot \log \frac{\gr L {f(S)} N} C
    \\
    & \geq
    \frac 1 N \cdot \log \gr L {f(S)} N - \varepsilon
    \\
    & \geq
    \log \egrs L {f(S)} - \varepsilon,
  \end{align*}
  as desired.
\end{proof}

\begin{rema}\label{rem:limegr_gen}
  The proof of Theorem~\ref{thm:limegr} is mainly based on properties
  of the limit action on the real tree. In fact, the theorem also
  holds under the following weaker assumptions on~$\Gamma$
  \parencite[Proposition~3.2]{f}: The group~$\Gamma$ is finitely
  generated, equationally Noetherian, and admits a non-elementary
  isometric action on a hyperbolic graph~$X$ that satisfies a uniform weak
  proper discontinuity condition \parencite[Definition~2.1]{f} and
  that admits a constant~$N$ such that for every~$S \in
  \FG(\Gamma)$, the set $S^N$ contains an element that acts
  hyperbolically on~$X$.
\end{rema}

%%%%%%%%%%%%
\subsection{Well-orderedness}

\begin{proof}[Sketch of proof of Theorem~\ref{thm:wellordered}]
  If the given hyperbolic group~$\Gamma$ is virtually cyclic,
  then $\EGR(\Gamma) = \{1\}$, which clearly is well-ordered.

  In the following, we consider the case when $\Gamma$ is
  non-elementary hyperbolic. We assume for a contradiction
  that there exists a sequence $(S_n)_{n\in \N}$ of
  finite generating sets of~$\Gamma$ such that $(\egrs \Gamma
  {S_n})_{n \in \N}$ is strictly monotonically decreasing.
  In particular, the sequence~$(\egrs \Gamma {S_n})_{n \in \N}$ is
  bounded. In view of the compactness theorem (Theorem~\ref{thm:compactness})
  and the invariance of the exponential growth rates under conjugation,
  we may assume without loss of generality
  that there exists a free group~$F$ with free generating set~$S$
  and epimorphisms~$(f_n \colon F \to \Gamma)_{n\in \N}$ with~$f_n(S) = S_n$
  and such that $f_*$ converges to a limit group~$f \colon F \to L$
  over~$\Gamma$ with a faithful action on a real tree. 

  We therefore obtain from Theorem~\ref{thm:limegr} that
  \begin{align*}
  \lim_{n \to \infty} \egrs \Gamma {S_n}
  & 
  = \lim_{n \to \infty} \egrs \Gamma {f_n(S)}
  \\
  & = \egrs L {f(S)}
  & \text{(Theorem~\ref{thm:limegr})}
  \\
  & \geq \egrs \Gamma {S_N}
  > \egrs \Gamma {S_{N+1}}
  > \dots,
  & \text{(for all~$N\gg 0$; Theorem~\ref{thm:limegr})}
  \end{align*}
  which is impossible. This contradiction shows that no such strictly
  decreasing sequence exists and hence $\EGR(\Gamma)$ is
  well-ordered.
\end{proof}

%%%%%%%%%%%%
\subsection{Finite ambiguity}

\begin{proof}[Sketch of proof of Theorem~\ref{thm:fin}]
  Let $r \in \R_{>1}$ and let us assume for a contradiction that
  there exists a sequence~$(S_n)_{n \in \N}$ of finite generating sets 
  that all represent different $\Aut(\Gamma)$-orbits and that
  satisfy~$\egrs \Gamma {S_n} = r$ for all~$n \in \N$. 

  Proceeding as before, by the compactness theorem
  (Theorem~\ref{thm:compactness}), we may assume without loss of
  generality that there exists a free group~$F$ with free generating
  set~$S$ and epimorphism~$(f_n \colon F \to \Gamma)_{n\in \N}$
  with~$f_n(S) = S_n$ and such that $f_*$ converges to a limit
  group~$f \colon F \to L$ over~$\Gamma$ with a faithful action on a
  real tree. Hence, Theorem~\ref{thm:limegr} shows that
  \[ \egrs L {f(S)} = \lim_{n \to \infty} \egrs \Gamma {f_n(S)} = r
  \]
  
  On the other hand, by passing to a subsequence, we may furthermore
  assume that for all~$n \in \N$, there exists a homomorphism~$h_n
  \colon L \to \Gamma$ with~$f_n = h_n \circ f$ (as in the proof of
  Theorem~\ref{thm:limegr}), that at most one of the
  epimorphisms~$h_n$ is an isomorphism (because the~$S_n$ lie in
  different $\Aut(\Gamma)$-orbits), and that the kernels of the~$h_n$
  contain no torsion. Then a careful refinement of the proof of
  Theorem~\ref{thm:limegr} shows that the strict inequality 
  \[ \egrs L {f(S)} > \egrs \Gamma {f_n(S)} = r
  \]
  holds for all~$n \in \N$ \parencite[Proposition~3.2]{fs}.
  This contradicts the previous computation that~$\egrs L {f(S)} = r$.
\end{proof}

%%%%%%%%%%%%
\subsection{Growth ordinals}

\begin{proof}[Sketch of proof of Theorem~\ref{thm:ord}]
  It suffices to show that $\egrord \Gamma \geq
  \omega^m$ for every~$m \in \N$. 
  Let $m \in \N$. We consider the sequence
  \[
  L_1 \coloneqq  \Gamma * F_m
  \to
  L_2 \coloneqq  \Gamma * F_{m-1}
  \to
  \dots
  \to
  L_m \coloneqq  \Gamma *\Z
  \to
  L_{m+1} \coloneqq  \Gamma
  \]
  of epimorphisms, where $F_j$ is a free group of rank~$j$
  and where the epimorphisms successively kill free generators
  and keep the $\Gamma$-factor intact. It helps to think of~$j$
  as the number of cusps.

  Let us first focus on a single step: If $\Lambda$ is a non-elementary
  hyperbolic group, then there exists a stable homomorphism~$(f_{n}
  \colon \Lambda * \Z \to \Lambda)_{n \in \N}$ consisting of
  epimorphisms that converges to~$\Lambda * \Z$.  Let $S$ be a finite
  generating set of~$\Lambda$ and let $\widetilde S \subset \Lambda *
  \Z$ be a generating set of~$\Lambda * \Z$, e.g., obtained by adding a
  free generator of~$\Z$. By passing to subsequences of~$f_*$, one
  can achieve the following strict monotonicity:
  \begin{itemize}
  \item The sequence~$(\egrs \Lambda {f_n(\widetilde S)})_{n \in \N}$
    is increasing and converges to~$\egrs {\Gamma *\Z}
    {\widetilde S}$; 
    this uses Theorem~\ref{thm:compactness} and Theorem~\ref{thm:limegr},
    as before.
  \item The values in the sequence are all different; this uses a
    finite ambiguity theorem for finitely generated subgroups of limit
    groups over~$\Lambda$ \parencite[Theorem~5.8]{fs}.
  \end{itemize}
  
  For notational simplicity, we now restrict to the case~$m = 2$.  We
  choose a finite generating set~$S$ of~$\Gamma$ and take the extended
  finite generating set~$\widetilde S$ of~$L_1 = \Gamma * F_2$.
  Applying the single step to~$L_1 \to L_2$ leads to a stable
  homomorphism~$f^1_*$ with strict monotonicity.  Let $f^2_*$ be a
  stable homomorphism for~$L_2 \to L_3$.  For each~$n_1 \in
  \N$, we apply the single step to~$L_2 \to L_3$ and the generating
  set~$f_{n_1}(\widetilde S)$ to select a subsequence of~$f^2_*$ with
  strict monotonicity. By composing with~$f_{n_1}$, we obtain a
  sequence~$f_{n_1,*}$ from~$L_1$ to~$L_3 = \Gamma$ such that $(\egrs
  {\Gamma} {f_{n_1,n}(\widetilde S)})_{n \in \N}$ is strictly
  increasing and converges to~$\egrs {L_2} {f_{n_1}(\widetilde S)}$.
  By varying~$n_1$, we thus see that $\egrord \Gamma \geq \omega^2$.

  For higher values of~$m$, one iterates these considerations
  appropriately.
\end{proof}

To prove~$\egrord \Gamma \leq \omega^\omega$ under
additional hypotheses, \textcite[proof of Theorem~4.2]{fs} construct
proper epimorphism chains of limit groups over~$\Gamma$ from
convergent sequences of convergent sequences of etc\dots\ of exponential
growth rates of~$\Gamma$; the Krull dimension property then gives 
control on the maximal lengths of such chains, whence on the maximal
powers of~$\omega$ that appear below a given threshold.

%%%%%%%%%%%%%%%%%%%%%%%%%%%%%%%%%%%%%%%%%%%%%%%%%%%%%%%%%%%%%%%%%%%%%%
\section{Applications and extensions}\label{sec:apps}

The well-orderedness of exponential growth rates
(Theorem~\ref{thm:wellordered}) in particular contains the fact that
all non-elementary hyperbolic groups have uniformly exponential
growth.

%%%%%%%%%%%%
\subsection{Hyperbolic groups are Hopfian}

A group~$\Gamma$ is \emph{Hopfian} if every self-epimorphism~$\Gamma
\to \Gamma$ is an automorphism. This property has applications in the
context of degrees of self-maps of closed manifolds.  Hyperbolic
groups are known to be Hopfian
\parencite{zela_endo,weidmannreinfeldt}. Using that the exponential
growth rates of hyperbolic groups are well-ordered,
\textcite[Corollary~2.9]{fs} complete an approach to proving that
hyperbolic groups are Hopfian outlined by
\textcite{delaharpe_exp}; this is not an independent alternative proof
because the current proof of Theorem~\ref{thm:wellordered} uses the very
results on limit groups that go into the previous proofs that
hyperbolic groups are Hopfian.

\begin{coro}%[\protect{\cite[Corollary~2.9]{fs}}]
  Every hyperbolic group is Hopfian.
\end{coro}
\begin{proof}
  Elementary hyperbolic groups are Hopfian because
  they are virtually cyclic (whence finitely generated
  and residually finite).
  
  Let $\Gamma$ be a non-elementary hyperbolic group and let $f \colon
  \Gamma \to \Gamma$ be an epimorphism. Because $\EGR(\Gamma)$
  is well-ordered (Theorem~\ref{thm:wellordered}), there exists
  a finite generating set~$S$ of~$\Gamma$ with
  $\egr \Gamma = \egrs \Gamma S.
  $
  \emph{Assume} for a contradiction that the kernel of~$f$
  is non-trivial. Then, \textcite{arzhantsevalysenok_tight}
  show that there is a \emph{strict} monotonicity 
  \[ \egrs \Gamma S > \egrs \Gamma {f(S)}.
  \]
  However, this contradicts the minimality property of~$S$.
  Thus, $f$ is an automorphism.
\end{proof}

All finitely generated residually finite groups are Hopfian. While
fundamental groups of closed hyperbolic manifolds are residually
finite and hyperbolic, it is a long-standing open problem whether all
hyperbolic groups are residually finite.

%%%%%%%%%%%%
\subsection{Generalisations}

The methods discussed in Section~\ref{sec:proof} by \textcite{fs}
extend to cover also the following generalisations:
\begin{itemize}
\item If $\Gamma$ is a hyperbolic group, then the set
  \[ \{ \egrs {H} {S}
  \mid H < \Gamma \text{ finitely generated and non-elementary},\ S \in \FG(H)
  \}
  \]
  is well-ordered \parencite[Theorem~5.1]{fs}. This
  can be viewed as an addition to the Tits alternative
  for hyperbolic groups. 
\item Moreover, in this subgroup setting, there is a corresponding
  finite ambiguity statement for non-elementary hyperbolic
  groups \parencite[Theorem~5.3]{fs}.  
\end{itemize}
As a consequence, they also obtain analogous results for limit groups
over non-elementary hyperbolic groups
\parencite[Corollary~5.6--5.10]{fs}.  Furthermore, the approach is
robust enough to admit an extension to the case of sub-semigroups
\parencite[Section~6]{fs}.

\textcite{f} adapted the method to obtain well-orderedness of
exponential growth rates sets for other classes of groups, including
certain groups acting acylindrically on hyperbolic spaces, rank-$1$
lattices, fundamental groups of strictly negatively curved Riemannian
manifolds, and certain relatively hyperbolic groups.  These results
can, for instance, be applied to certain subgroups of right-angled
Artin groups \parencite[Corollary~1.0.11]{kerr}.

%%%%%%%%%%%%%%%%%%%%%%%%%%%%%%%%%%%%%%%%%%%%%%%%%%%%%%%%%%%%%%%%%%%%%%
%%%%%%%%%%%%%%%%%%%%%%%%%%%%%%%%%%%%%%%%%%%%%%%%%%%%%%%%%%%%%%%%%%%%%%
\appendix

%%%%%%%%%%%%%%%%%%%%%%%%%%%%%%%%%%%%%%%%%%%%%%%%%%%%%%%%%%%%%%%%%%%%%%
\section{Terminology}\label{appx:terminology}

For the sake of completeness, we recall the basic terminology
appearing in the main results (Section~\ref{sec:results}).

%%%%%%%%%%%%
\subsection{Hyperbolic groups}

Finitely generated groups are hyperbolic if their Cayley graphs
are ``negatively curved'' in the sense that geodesic triangles
in are uniformly slim \parencite{gromov_hypgroups,bh}:

\begin{defi}[hyperbolic group]
  A finitely generated group~$\Gamma$ is \emph{hyperbolic} if the
  Cayley graph of~$\Gamma$ with respect to one (whence every
  \parencite[Theorem~III.H.1.9]{bh}) finite generating set is a
  hyperbolic metric space. A hyperbolic group is \emph{non-elementary}
  if it is not virtually cyclic.
\end{defi}

\begin{exem}[hyperbolic groups]
  Fundamental groups of closed smooth manifolds that admit a
  Riemannian metric of negative sectional curvature are
  hyperbolic in view of the \v Svarc--Milnor lemma and
  the fact that CAT$(-1)$-spaces are hyperbolic metric
  spaces. In particular, this includes
  the fundamental groups of closed hyperbolic manifolds.
  Such fundamental groups are virtually cyclic if and only
  if the dimension is at most~$1$.

  Finitely generated free groups are hyperbolic.  
  The class of hyperbolic groups is closed under quasi-isometries
  \parencite[Theorem~III.H.1.9]{bh} and under certain amalgamations
  \parencite{bestvinafeighn}.
\end{exem}

The group~$\Z^2$ is \emph{not} hyperbolic.  More generally, all
finitely generated groups that contain a subgroup isomorphic to~$\Z^2$
are \emph{not} hyperbolic \parencite[Corollary~III.$\Gamma$.3.10]{bh}. 
In general, subgroups of hyperbolic groups need \emph{not} be
hyperbolic.

%%%%%%%%%%%%
\subsection{Exponential growth rates of groups}

The exponential growth rate of groups measures the exponential
expansion rate of the size of balls in Cayley graphs
\parencite[Chapter~VII.B]{delaharpe}:

\begin{rema}\label{rem:fekete}
  Let $\Gamma$ be a finitely generated group and let $S \subset \Gamma$
  be a finite generating set of~$\Gamma$. We write~$\gr \Gamma S n$
  for the number of elements in the $n$-ball~$\Gr \Gamma S n$
  of the Cayley graph of~$\Gamma$ with respect to~$S$. Then
  $\gr \Gamma S {n+m} 
  \leq \gr \Gamma S n \cdot \gr \Gamma S m  
  $ 
  for all~$n, m \in \N$. Therefore, the Fekete lemma 
  shows that the limit of~$(\gr \Gamma S n ^{1/n})_{n \in \N}$ exists
  and that
  \[ \lim_{n \to \infty} \gr \Gamma S n ^{1/n}
  = \inf_{n \in \N_{>0}} \gr \Gamma S n ^{1/n}.
  \]
\end{rema}

\begin{defi}[exponential growth rate]
  Let $\Gamma$ be a finitely generated group. 
  \begin{itemize}
  \item Let $S \subset \Gamma$ be a finite generating
    set of~$\Gamma$. The \emph{exponential growth rate of~$\Gamma$
      with respect to~$S$} is defined as
    \[ \egrs \Gamma S \coloneqq  \lim_{n \to\infty} \gr \Gamma S n ^{1/n}.
    \]
  \item We write
    $\EGR(\Gamma) \coloneqq  \{ \egrs \Gamma S \mid S \in \FG(\Gamma) \}
    $
    for the (countable) set of all exponential growth rates of~$\Gamma$, 
    where $\FG(\Gamma)$ denotes the set of finite generating sets
    of~$\Gamma$.
  \item The \emph{exponential growth rate of~$\Gamma$} is
    the infimum
    \[ \egr \Gamma \coloneqq  \inf \EGR(\Gamma).
    \]
  \item
    The group~$\Gamma$ has \emph{exponential growth} if there exists
    an~$S \in \FG(\Gamma)$ with~$\egrs \Gamma S > 1$.  The
    group~$\Gamma$ has \emph{uniform exponential growth} if~$\egr
    \Gamma > 1$.
   \end{itemize}
\end{defi}

\begin{rema}[monotonicity of exponential growth rates]\label{rem:mono}
  Let $\Gamma$ and $\Lambda$ be finitely generated groups.
  \begin{enumerate}
  \item If $f \colon \Gamma \to \Lambda$ is an epimorphism, and $S \in
    \FG(\Gamma)$, then
    $\egrs \Gamma S \geq \egrs \Lambda {f(S)}.
    $ 
  \item
    If $\Lambda$ is a subgroup of~$\Gamma$ and $\Lambda$ has
    exponential growth, then also $\Gamma$ has exponential growth.
  \end{enumerate}
\end{rema}

\begin{exem}[exponential growth]
  Finitely generated free groups have exponential growth if and only
  if they are of rank at least~$2$. More precisely
  \parencite[Proposition~VII.13]{delaharpe}: If $F$ is a free group
  and $S \in \FG(F)$, then
  \[ \egrs F S \geq 2 \cdot \rk(F) - 1.
  \]

  Because non-elementary hyperbolic groups [uniformly] contain free
  groups of rank~$2$, monotonicity shows that they have [uniform]
  exponential growth \parencite{koubi}.
    
  There exist finitely generated groups that have 
  exponential growth but do \emph{not} have uniform exponential
  growth \parencite{wilson}. In particular, for such groups~$\Gamma$,
  the set~$\EGR(\Gamma)$ is \emph{not} well-ordered.
\end{exem}

Exponential growth rates seem to be fragile under quasi-isometries: It
is an open problem to determine whether uniform exponential growth is
stable under quasi-isometries.

%%%%%%%%%%%%
\subsection{Well-ordered countable sets and ordinals}

Well-orderings are orderings that allow for induction
principles. Moreover, well-orderings admit an arithmetic,
the ordinal arithmetic.

\begin{defi}[well-ordered sets, ordinals]
  An ordered set~$(A,<)$ is \emph{well-ordered} if every
  non-empty subset of~$A$ contains a $<$-minimal element.
  An \emph{ordinal} is an isomorphism class of well-ordered
  ordered sets. An ordinal is \emph{countable} if the underlying
  set is countable.
\end{defi}

In the context of Section~\ref{sec:results}, the following ordinals
are important (Figure~\ref{fig:ordinals}):

\def\ordpt#1{%
  \fill #1 circle (0.07cm);}
\def\ordomega#1#2{%
  \begin{scope}[x=#2cm]
    \foreach \j in {1,...,8} {%
      \begin{scope}[shift={(1-1/\j,0)},scale={3/\j^2},every node/.style={scale=1/\j^2}]
        #1
      \end{scope}
    }
  \end{scope}}
\def\ordbox#1{%
  \draw[black!50] (-0.05/2,-0.1) rectangle +(0.18,0.2);
  \draw (0,0) node {$#1$};}

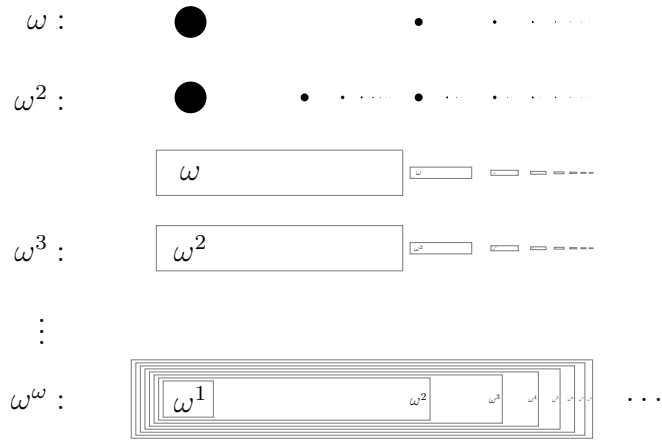
\begin{figure}
  \begin{center}
    \begin{tikzpicture}[x=1cm,y=1cm,thin]
      \begin{scope}[shift={(0,1)}]
        \draw (-1.5,0) node[anchor=east] {$\omega:$};
        \ordomega{\ordpt{(0,0)}}{6}
      \end{scope}
        
      \begin{scope}[shift={(0,0)}]
        \draw (-1.5,0) node[anchor=east] {$\omega^2:$};
        \ordomega{%
                  \begin{scope}[scale={2/6}]
                    \ordomega{\ordpt{(0,0)}}{3}
                  \end{scope}}
                 {6}
      \end{scope}

      \begin{scope}[shift={(0,-1)}]
        %\draw (-1.5,0) node[anchor=east] {$\omega^2:$};
        \ordomega{\ordbox{\omega}}{6}
      \end{scope}

      \begin{scope}[shift={(0,-2)}]
        \draw (-1.5,0) node[anchor=east] {$\omega^3:$};
        \ordomega{\ordbox{\omega^2}}{6}
      \end{scope}

      \begin{scope}[shift={(0,-3)}]
        \draw (-1.5,0) node[anchor=east] {$\vdots$\phantom{\ :}};
      \end{scope}

      \begin{scope}[shift={(0,-4)}]
        \draw (-1.5,0) node[anchor=east] {$\omega^\omega:$};
        \begin{scope}[x=6cm]
          \foreach \j in {1,...,8} {%
            \begin{scope}[shift={(1-1/\j,0)}]
              \draw[black!50] (1/\j-1-0.05-0.01*\j,-0.2-0.04*\j)
              rectangle +(1-1/\j+0.05+0.01*\j + 0.05/\j,0.4+0.08*\j);
              \begin{scope}[scale={1.5/\j},every node/.style={scale=1/\j}]
                \draw (0,0) node {$\omega^{\j}$};
              \end{scope}
            \end{scope}
          }
          \draw (1,0) node {$\dots$};
          \end{scope}
      \end{scope}

    \end{tikzpicture}
  \end{center}

  \caption{The ordinals~$\omega$, $\omega^2$, \dots, $\omega^\omega$, schematically}
  \label{fig:ordinals}
\end{figure}

\begin{exem}[$\omega^\omega$]
  The natural numbers~$\N$ are well-ordered with respect to the
  standard order. The corresponding ordinal is denoted~$\omega$.
  For~$k \in \N$, we write~$\omega^k$ for the ordinal represented
  by~$\N^k$ with the lexicographic order.  Equipping the finite
  support functions~$\N \to \N$ with the lexicographic order leads to
  a well-ordered set; its ordinal number is denoted
  by~$\omega^\omega$. The ordinal~$\omega^\omega$ can alternatively
  also be described as~$\sup_{k \in \N} \omega^k$.
\end{exem}

\begin{exem}  
  The subsets $A \coloneqq  \{ 1-1/n \mid n \in \N_{>0}\}$ and $B\coloneqq  \bigcup_{n
    \in \N} (n+A)$ of~$\R$ are well-ordered with respect to the
  standard order on~$\R$.  The set~$A$ represents the
  ordinal~$\omega$ and $B$ represents the ordinal~$\omega^2$.
  The subset~$\Q_{\geq 0} \subset \R$ is \emph{not} well-ordered.
\end{exem}

%%%%%%%%%%%%%%%%%%%%%%%%%%%%%%%%%%%%%%%%%%%%%%%%%%%%%%%%%%%%%%%%%%%%%%
\section{Right-computability of exponential growth rates}\label{appx:rc}

We provide proofs for the right-computability claims in
Section~\ref{subsec:rc}. 

\begin{proof}[Proof of Proposition~\ref{prop:egr_algorc}]
  Let $F(S)$ be the set of reduced words over~$S \sqcup S^{-1}$.
  In particular, $F(S)$ is a free group, freely generated by~$S$,
  with respect to the composition given by concatenation and reduction.
  It is well known that we can Turing-enumerate all finite subsets of~$F(S)$
  that represent generating sets of~$\Gamma$ under the canonical
  projection~$F(S) \to \genrel SR = \Gamma$ (by Turing-enumerating 
  the normal closure of~$R$ in~$F(S)$). 

  Therefore it suffices to show that there exists a Turing machine that
  given a finite generating set~$S'
  \subset F(S)$ of~$\Gamma$ enumerates the set~$A(S') \coloneqq  \{ x \in \Q \mid x > \egrs
  \Gamma {S'} \}$. 
  By the Fekete lemma (Remark~\ref{rem:fekete}), for all generating
  sets~$S'$, we have
  \[
  A(S')
  =
  \bigl\{ x \in \Q
  \bigm| \exi{n\in \N_{>0}} x^n > \gr \Gamma {S'} n
  \bigr\}.
  \]
  The numbers~$\gr \Gamma {S'} n$ are not necessarily computable
  in terms of~$n$ and~$S'$ (as the word problem might not be solvable in~$\Gamma$),
  but recursively enumerating the normal closure of~$R$ in~$F(S)$
  shows that there exists a Turing machine that given a finite generating
  set~$S' \subset F(S)$ of~$\Gamma$ enumerates 
  $\{ (n,m) \mid n,m \in \N,\ m \geq \gr \Gamma {S'} n \};
  $ 
  hence, there is also a Turing machine for~$A(\;\cdot\;)$. 
\end{proof}

\begin{proof}[Proof of Corollary~\ref{cor:egr_basicrc}]
  The first part is a direct consequence of Proposition~\ref{prop:egr_algorc}. 

  For the second part, let $r \in \Q$. For all~$S \in \FG(\Gamma)$, we have
  \[ \egrs \Gamma {S} < r
  \iff
  \exi{x \in \Q} (r > x \land x > \egrs \Gamma S).
  \]
  Let $\genrel SR$ be a finite presentation of~$\Gamma$.
  Using a Turing machine as provided by
  Proposition~\ref{prop:egr_algorc}, we can thus
  construct a Turing machine that enumerates
  all finite sets~$S'$ of words over~$S \sqcup S^{-1}$
  such that $S'$ represents a generating set of~$\Gamma$
  and such that $\egrs \Gamma {S'} < r$.
\end{proof}

For simplicity, we restricted the discussion to finitely presented
groups. Similar arguments also apply to finitely generated recursively
presented groups.
Conversely, one might wonder whether every right-computable real
number~$\geq 1$ can be realised as the exponential growth rate of some
finite generating set of some finitely/recursively presented group.

%%%%%%%%%%%%%%%%%%%%%%%%%%%%%%%%%%%%%%%%%%%%%%%%%%%%%%%%%%%%%%%%%%%%%%
\subsection*{Acknowledgements}

This work was supported by the CRC~1085 \emph{Higher Invariants}
(Universit\"at Regensburg, funded by the DFG).

%%%%%%%%%%%%%%%%%%%%%%%%%%%%%%%%%%%%%%%%%%%%%%%%%%%%%%%%%%%%%%%%%%%%%%
%% printbibliography is the command from the package biblatex
\printbibliography

\end{document}
%%%%%%%%%%%%%%%%%%%%%%%%%%%%%%%%%%%%%%%%%%%%%%%%%%%%%%%%%%%%%%%%%%%%%%
%%%%%%%%%%%%%%%%%%%%%%%%%%%%%%%%%%%%%%%%%%%%%%%%%%%%%%%%%%%%%%%%%%%%%%

%%% Local Variables:
%%% mode: latex
%%% TeX-master: t
%%% End: